\documentclass{article}
\usepackage{graphicx}
\usepackage{amsmath}
\usepackage{psfrag}
\usepackage{amsfonts}
\usepackage{amssymb}
\usepackage{pb-diagram}
\usepackage{hyperref}

\newtheorem{theorem}{Theorem}

\newtheorem{corollary}[theorem]{Corollary}

\newtheorem{definition}[theorem]{Definition}

\newtheorem{fact}[theorem]{Fact}
\newtheorem{lemma}[theorem]{Lemma}

\newtheorem{proposition}[theorem]{Proposition}

\newenvironment{proof}[1][Proof]{\textbf{#1.} }{\ \rule{0.5em}{0.5em}}

\title{A generalized closure concept based on neighborhood-equivalence and preserving graph Hamiltonicity}
\bigskip
\author{Thierry Vall\'ee\\[6pt] {\small 7 All\'ee Georges Rouault, 75020 Paris}\\ {\small Phone: +33 (0) 7 87 16 68 88}\\ \small{Email: vallee\_th@yahoo.fr; vallee@pps.univ-paris-diderot.fr}} 
\date{}

\begin{document}
\maketitle

\begin{abstract}
A graph is Hamiltonian if it contains a cycle which goes through all vertices exactly once. Determining if a graph is Hamiltonian is known as
a NP-complete problem and no satisfactory characterization for these graphs has been found. 

In 1976 Bondy and Chv\`atal introduced a way to get round the Hamiltonicity problem complexity by using a closure of the graph. This closure is a supergraph of G which preserves Hamiltonicity, that is, which is Hamiltonian if and only if G is. Since this seminal work, several closure concepts preserving Hamiltonicity were introduced. In particular Ryj\`acek defined in 1997 a closure concept for claw-free graphs based on local completion. The completion is performed for every eligible vertex of the graph.

Extending these works, Bretto and Vall\'ee recently introduced a new closure concept preserving Hamiltonicity and based on local completion. The local completion is performed for each neighborhood-equivalence eligible vertex of the graph. 

In this article, we generalize the main results of Bretto and Vall\'ee by introducing a broader notion of neighborhood-equivalence eligibility, allowing the definition of a denser graph closure which still preserves Hamiltonicity.
\end{abstract}


\section{Introduction}

A graph is Hamiltonian if it contains a cycle which goes through all vertices exactly once. Determining if a graph is Hamiltonian is known as
a NP-complete problem and no satisfactory characterization for these graphs has been found. A huge body of literature exists on the subject surveyed for instance in \cite{Gould1991,Gould2003}.

In \cite{BondyChvatal}, Bondy and Chv\`atal introduced a way to get round the Hamiltonicity problem complexity by using a closure of the graph. The closure is then proved to be Hamiltonian if and only if the graph is. In particular, if the closure is a complete graph then the graph is Hamiltonian. Since this seminal article, several closure concepts preserving Hamiltonicity were introduced (for a survey on the topic, see for instance \cite{BroersmaRyjacekSchiermeyer}). In particular Z. Ryj\`acek defined in \cite{Ryjacek1} a closure concept for claw-free graphs based on local completion. The local completion is repeatedly performed on every \emph{eligible} vertex, as long as such a vertex exists.

Following a different approach, Goodman and Hedetniemi gave in \cite{GooHed1974} a sufficient condition for Hamiltonicity based on the existence of a clique-covering of the graph. This condition was recently generalized in \cite{BreVal1,BreVal2} using the notion of \emph{Eulerian} clique-covering. It was also shown in \cite{Vallee1}, that there exists an Eulerian clique-covering of a graph if and only if there exists a \emph{normal} one, where a clique-covering is normal if it contains the closed neighborhood of every simplicial vertex of the graph. In this context, closure concepts based on local completion are interesting since, then, the closure of a graph contains more simplicial vertices than the graph itself, making the search for a normal clique-covering easier. For instance, a closure in the sense of \cite{Ryjacek1} has at most one \emph{normal} Eulerian clique-covering.

In \cite{BreVal3} a new closure concept based on local completion and preserving Hamiltonicity for all graphs is studied. The closure is defined using the notion of \emph{neighborhood-equivalence} as first introduced in \cite{BreVal1}, and is obtained by performing a local completion at all \emph{neighborhood-equivalence eligible (N-eligible)} vertices of the graph.\smallskip 

In the sequel, we generalize N-eligibility using the notion \emph{2-weighted N-eligibility (N2-eligibility)} (cf. Definition \ref{DfN2-Eligible}). We show how to obtain the \emph{N2-closure} of a graph by performing recursively a local completion at a chosen \emph{N2-eligible} vertex. The N2-closure is then proved to be Hamiltonian if and only if the graph is. In particular, the N2-closure is a supergraph of the N-closure whenever the N-eligible vertices are chosen first during the completion process.

In a first section, we introduce some notations and remind the reader of the definitions of local completion and neighborhood-equivalence. In a second section, we introduce \emph{alternating} paths and present some of their properties. Finally, in Section \ref{SecMain}, these paths are used to show  that, for all graph $G$ and choice function $\rho$ on the set of vertices of $G$, the N2-closure of $G$ is a N2-eligible free graph which circumference is equal to the circumference of $G$. We conclude by giving an example (Figure \ref{CaptureN2EligibleFig}) which shows that there are in general more than one N2-closure for a given graph.  

\section{Preliminaries}

\subsection{General Notations}

In the sequel, $|X|$ denotes the cardinal of the set $X$, and $X \setminus Y =\{ x \in X : x \notin Y\}$. We also define $\mathbb{P}(X) = \{\{x,y\} \subseteq X : x \neq y\}$ and $[X, Y] = \{\{x,y\} \in \mathbb{P}(X \cup Y):  x \in X, \, y \in Y\}$. Notice that $[X, Y]= [Y, X]$ and that $X \subseteq X'$ implies $[X, Y] \subseteq [X', Y]$ and $[Y, X] \subseteq [Y, X']$.

We always suppose a graph to be undirected, simple and finite. Thus a graph $G$ is a pair $(V,E)$ where $V$ is the vertex set of $G$, and $E$ is a subset of $\mathbb{P}(V)$. To simplify notations a pair $\{x,y\} \in \mathbb{P}(V)$ is simply written $xy$. If $X \subseteq V$, $E(X) = \{ xy \in E : x,y \in X\}$.

The \emph{(open) neighborhood of $x \in V$} is the set $N(x) = \{y: xy \in E\}$. Its \emph{closed neighborhood} is the set $N[x] = N(x) \cup \{x\}$. The \emph{closed neighborhood of $X \subseteq V$} is the set $N[X] = \cup_{x \in X} N[x]$. Its \emph{(open) neighborhood} is the set $N(X) = N[X] \setminus X$.\smallskip

A walk in $G$ is a sequence of vertices $w=x_0 \dotsc x_k$ such that $x_ix_{i+1} \in E$, for every $i \in \{0, \dotsc , k\! - \!1\}$. The integer $k$ is the length of $w$, $x_0,x_k$ are its endpoints, $x_0$ its starting point and $x_k$ its ending point. In particular $x$ is a walk of length $0$ with starting and ending point $x$. A walk is \emph{closed} if $k \geq 3$ and $x_0 =x_k$. A closed walk is a \emph{cycle} if it contains no repetition of vertex except for the endpoint. We denote by $c(G)$ the circumference of $G$, that is, the length of the longest cycle in $G$.  A cycle is \emph{Hamilton} if it contains every vertex of the graph. A graph is \emph{Hamiltonian} if it contains a Hamilton cycle. 

If $w= x_0 \dotsc x_k$ is a walk then $\mathbb{V}(w) = \{x_0, \dotsc , x_{k}\}$ and $\mathbb{E}(w) = \{x_{i}x_{i+1} : i \in \{0, \dotsc , k\! - \!1\}\,\}$. Notice that there are infinitely many graphs in which a given sequence $w$ is a walk and that $\mathbb{V}(w) \subseteq V$ and $\mathbb{E}(w)\subseteq E$, for every such graph. 

If $w= x_0 \dotsc x_k$ is a walk then $x_i\stackrel{\rightarrow}{w} x_{j}$, where $0 \leq i \leq j \leq k$, denotes the subwalk of $w$ with endpoints $x_i,x_j$, and $x_j\stackrel{\leftarrow}{w} x_{i}$ the reverse walk $x_{j} x_{j-1} \dotsc  x_i$. In particular, if $i =j$ then $x_i\stackrel{\rightarrow}{w} x_{j} = x_i$ and we let $\stackrel{\leftarrow}{w} = x_k  \dotsc x_0$. Notice that $w = x_0\stackrel{\rightarrow}{w} x_k$. For every $i \in \{0, \dotsc , k\!- \!1\}$, $x_i^+$ is the successor of $x_i$ in $w$, that is, $x_i^+ = x_{i+1}$. For every $i \in \{ 1, \dotsc , k\}$, $x_i^-$ is the predecessor of $x_i$ in $w$, that is, $x_i^- = x_{i-1}$. Finally, if $w' = y_0 \dotsc y_n$ is a walk then $ww'$ denotes the sequence $x_0 \dotsc x_k y_0 \dotsc y_n$. Clearly $ww'$ is a walk if and only if $x_ky_0 \in E$.

If $C = x_0 \dotsc x_kx_0$ is a cycle then, for every $i \in \{1,\dotsc ,k\}$, the walk $x_i\stackrel{\rightarrow}{C}x_kx_0\stackrel{\rightarrow}{C}x_i$ is also a cycle. Such a cycle is said to be a \emph{rotation} of $C$. Clearly each rotation of $C$ has the same vertices than $C$ and so is of maximal length if and only if $C$ is. The same remark applies to $\stackrel{\leftarrow}{C}$.\smallskip

A \emph{path} is a walk containing no repetition of vertex. For every path $P$, we define the \emph{neighborhood of $x$ in $P$} as the set $P(x)$ which contains, when defined, the predecessor and the successor of $x$. That is, if $P =x$ then $P(x) = \emptyset$; and if $k\neq 0$ and $P = x_0 \dotsc x_k$ then $P(x_0) = \{x^+_0\}$, $P(x_k) = \{x^-_k\}$ and $P(x_i) = \{x^-_i, x^+_i\}$, for every $i \in \{1, \dotsc , k\! - \!1\}$.  Notice that clearly $P(x) \subseteq N(x) \cap \mathbb{V}(P)$ but, since $P(x)$ contains only the immediate predecessor and immediate successor of $x$ in $P$, we may have $P(x) \neq N(x) \cap \mathbb{V}(P)$. For every $X \subseteq \mathbb{V}(P)$, we let $P(X) = \cup_{x\in X} P(x)$. Notice that, contrarily to $N(X) \cap X$, $P(X) \cap X$ may not be empty.\smallskip

A set $X \subseteq V$ is a \emph{clique} of $G$ if $xy \in E$, for all distinct $x,y \in X$. A vertex is \emph{simplicial} if $N[x]$ is a clique of $G$. We denote by $S$ the set of simplicial vertices of $G$ and by $N \! S$ the set of non-simplicial vertices of $G$.\smallskip

For every $X\subseteq V$, the \emph{size $\sigma (X)$} of $X$ in $G$ is $|E(X)|$. Its \emph{size $\sigma_P(X)$ in $P$}, where $P$ is a path, is $| \mathbb{E}(P) \cap [X, X]|$.\smallskip

A graph $G$ is connected if all distinct vertices $x,y$ are connected by a walk, and complete if $E = \mathbb{P}(V)$. From now on, $G$ is always supposed to be connected.

\subsection{Local completion, neighborhood equivalence}

We define \emph{local completion} below using the notion of \emph{neighborhood-equivalence} as defined in \cite{BreVal1}. We also give some easy results.

\begin{definition} Two vertices $x,y$ of $G$ are \emph{neighborhood-equivalent} if $N[x] = N[y]$. We write $x \equiv y$ to express that $x,y$ are neighborhood equivalent and $\bar x$ is the class of $x$ modulo $\equiv$.
\end{definition}

\begin{fact}\label{FaBasicN-equivalence} For all $x \in V$:
\begin{enumerate}
\item\label{FaBasicN-equivalence.3} $N[\bar x] = N[x]$.
\item\label{FaBasicN-equivalence.1} $\bar x$ is a clique and $[N(\bar x), \bar x] \subseteq E$.
\end{enumerate}
\end{fact}

\begin{definition}\label{DfLocalCompletion} The \emph{local completion of $G$ at $x$} is the graph $G_x =(V,E_x)$, where $E_x = E \cup B_x$ and $B_x = \{ y z: y z \notin E, \, y,z \in N(\bar x) \}$.
\end{definition}

\noindent Obviously $E \cap B_x = \emptyset$. Moreover, if $y,z \in N[x]$ and $yz \notin E$ then $y,z \in N(\bar x)$ by Fact \ref{FaBasicN-equivalence}.\ref{FaBasicN-equivalence.1}. Hence, it is easy to see that $B_x = \{ y z: y z \notin E, \, y,z \in N[x]\}$ and so that $N[x]$ is complete in $G_x$. It is also clear that $G_x$ is connected if $G$ is.\smallskip

In the sequel, we denote respectively by $N_x[z]$ (resp. ${\bar z}^x$) the neighborhood (resp. the neighborhood-equivalence class of $z$) in $G_x$ and by $S_x$ (resp. $N\!S_x$) the set of simplicial (resp. non-simplicial) vertices of $G_x$.

\begin{fact}\label{FaBasicLocalCompletion} For every $x\in V$, $G$ is a spanning subgraph of $G_x$ and:
\begin{enumerate}
\item\label{FaBasicLocalCompletion.1.5} For every $x' \in \bar x$, $N[x'] = N_x[x']$.
\item\label{FaBasicLocalCompletion.1.6} $\bar x \subseteq S_x \supseteq S$.
\end{enumerate}
\end{fact}
\begin{proof} Clearly $G$ is a spanning subgraph of $G_x$ and so $N[y] \subseteq N_x[y]$ for every $y \in V$. Moreover, clearly, by definition of $B_x$, if $N[y] \neq N_x[y]$ then $y \in N(\bar x)$ and so $y \notin \bar x$. That proves the first point which in turn implies easily $\bar x \subseteq S_x$, since $N_x[x] = N[x]$ is a clique in $G_x$. Notice now that $N(\bar x) = N[x] \setminus \bar x$ by Fact \ref{FaBasicN-equivalence}.\ref{FaBasicN-equivalence.3}, and so $N(\bar x) \subseteq N(x)$.

It remains to show that $S \subseteq S_x$. So let $y \in S$. If $N[y] = N_x[y]$ then $y \in S_x$, since $N[y]$ is a clique in $G$ and so in $G_x$. If now $N[y] \neq N_x[y]$, we have $y \in N(\bar x) \subseteq N(x)$ and so $x \in N(y)$. Let now $u,v \in N_x[y]$, we must show $uv \in E_x$. If $uv \in E$ then the result is immediate, since $E \subseteq E_x$, so we can suppose $uv \notin E$. Hence, by simpliciality of $y$, at least one vertex among $u,v$ is not in $N[y]$. Without loss of generality, we can suppose $u$ this vertex. We have $y \in N_x[u] \setminus N[u]$, and so $N_x[u] \neq N[u]$ and $u \in N(\bar x) \subseteq N(x)$. If now $v \in N(\bar x)$ then $uv \in B_x$, by definition of $B_x$. Finally, if $v \notin N(\bar x)$ then $N_x[v] = N[v]$ and, since $v \in N_x[y]$, we get $y \in N[v]$ and so $v \in N[y]$. Hence, since $x \in N(y)$, we have $xv \in E$ by simpliciality of $Y$. So $v \in N[x]$ and, since $u\in N[x]$, we conclude $uv \in E_x$ from the fact that $N[x]$ is clique in $G_x$. Thus $N_x[y]$ is a clique in $G_x$ and so $y \in S_x$. That proves $S \subseteq S_x$ and so the second point.
\end{proof}


\begin{lemma}\label{LmClassXonC} Let $X \subseteq V$ and $y,z \in V$ such that $[\{y,z\}, \,X] \subseteq E$. If $C$ is a cycle of maximal length such that $yz\in \mathbb{E}(C)$ then $\mathbb{V}(C) \cap X = X$.
\end{lemma}
\begin{proof} Obviously $\mathbb{V}(C) \cap X \subseteq X$. Up to a rotation we can suppose that $y$ is the starting point of $C$, and so either $C =yz \stackrel{\rightarrow}{C} y$ or $C =y\stackrel{\rightarrow}{C} zy$. We can suppose the second case (otherwise the proof is done for $\stackrel{\leftarrow}{C}$). Moreover, if we suppose $x \in X \setminus \mathbb{V}(C)$ then, since $yx,zx \in [\{y,z\}, \, X] \subseteq E$, $y  \stackrel{\rightarrow}{C} zxy$ is a cycle in $G$ strictly longer than $C$, contradicting the maximality of $C$.
\end{proof}

\section{Alternating paths}

In this section, we assume a graph $G$ with set of vertices $V$ and set of edges $E$. 

\begin{definition}\label{DfAlternatingPath} Let $P$ be a path in $G$ and $X,Y$ be disjoint subsets of $\mathbb{V}(P)$. If the endpoints of $P$ are in $Y$ then:
\begin{enumerate}
\item $P$ is \emph{$Y \! X$-pseudo-alternating} if moreover $P(X) \subseteq X \cup Y$. 
\item $P$ is \emph{$Y \! X$-semi-alternating} if moreover $P(X) \subseteq Y$. 
\item $P$ is \emph{$Y \! X$-alternating} if moreover $P(X) \subseteq Y$ and $P(Y) \subseteq X$.
\end{enumerate}
\noindent A $Y\!X$-pseudo-alternating path is \emph{proper} if it is not $Y \! X$-semi-alternating. 

\noindent A $Y\!X$-semi-alternating path is \emph{proper} if it is not $Y \! X$-alternating.
\end{definition}

\noindent  Notice that a $Y \! X$-alternating path is a particular case of $Y \! X$-semi-alternating path which in turn is a particular case of $Y \! X$-pseudo-alternating path. Moreover, the $Y \! X$-pseudo-alternating path $P$ is proper if and only if $P(X) \cap X \neq \emptyset$. Finally, notice that $P=y$ is the unique $\{y\} \emptyset$-alternating path.

\begin{fact}\label{FaAlternatingPath} If $P$ is a path and $X,Y$ are disjoint subsets of $\mathbb{V}(P)$ then $P$ is $Y \! X$-alternating if and only if $P$ satisfies the following conditions:
\begin{enumerate} 
\item\label{FaAlternatingPath.1} $|X| = |Y|-1$
\item\label{FaAlternatingPath.2} There exist two enumerations $y_0, \dotsc ,y_n$ and $x_0, \dotsc , x_{n-1}$ of $Y$ and $X$ such that $P = y_0 x_0 \dotsc x_{n-1} y_n$.
\item\label{FaAlternatingPath.3} $\mathbb{V}(P) = X \cup Y$, $P(X) = Y$ and $P(Y) = X$.
\end{enumerate}
\end{fact}
\begin{proof} If is easy to check that if $P$ satisfies the conditions above then $P$ is $Y \! X$-alternating. Now if $P$ is $Y \! X$-alternating then, using the fact that the endpoints of $P$ are in $Y$, that $P(Y) \subseteq X$ and that $P(X) \subseteq Y$, it is easy to build the enumerations $y_0 \dotsc y_n$ and $x_0 \dotsc x_{n-1}$, where $y_0$ is the starting point of $P$, $x_0$ its successor, and so on, until we reach the ending point $y_n$ of $P$. It is then easy to check that $P$ satisfies the other conditions.
\end{proof}

\noindent We remind the reader that $\sigma_P(Y)$ is the size of $Y$ in $P$, that is, $\sigma_P(Y) =|\mathbb{E}(P) \cap [Y,Y]|$.

\begin{lemma}\label{LmGeneralYX} If $P$ is $Y \! X$-semi-alternating then:
\begin{enumerate}
\item\label{LmGeneralYX.1} $|X| < |Y| - \sigma_P(Y)$.
\item\label{LmGeneralYX.2} If \ $\mathbb{V}(P) \setminus (X \cup Y) \neq  \emptyset$ then $|X| < |Y|-1$.
\end{enumerate}
\end{lemma}
\begin{proof} Let $P = z_0 \dotsc z_k$ be a path verifying the conditions of the lemma. For every $i \in \{0, \dotsc , k\}$, let $P_i = z_0 \dotsc z_i$, $X_i = \mathbb{V}(P_i) \cap X_i$, $Y_i = \mathbb{V}(P_i) \cap Y_i$. Let also $Z_i = \mathbb{V}(P_i) \setminus X_i \cup Y_i$ and define $b_i =0$ if $Z_i = \emptyset$, and $b_i=1$ otherwise. Notice that $X_i$, $Y_i$ and $Z_i$ are pairwise disjoint, since $X$ and $Y$ are disjoint by Definition \ref{DfAlternatingPath}, and that $\mathbb{V}(P_i) = X_i \cup Y_i \cup Z_i$. We let $Z = \mathbb{V}(P) \setminus (X \cup Y)$.

We show first by induction on $i \in \{0 ,\dotsc ,k\}$ that:

\begin{enumerate}
\item If $z_i \in X$ then $|X_i| \leq |Y_i| - \sigma_P (Y_i)$ and $|X_i| \leq |Y_i| - b_i$.
\item If $z_i \in Y$ then $|X_i| < |Y_i| - \sigma_P (Y_i)$ and $|X_i| < |Y_i| - b_i$.
\item If $z_i \in Z$ then $|X_i| < |Y_i| - \sigma_P (Y_i)$ and $|X_i| \leq |Y_i| - b_i$.
\end{enumerate} 

If $i=0$ then the result comes easily from the fact that $z_0 \in Y$ (Definition \ref{DfAlternatingPath}). As induction hypothesis suppose now the result true for $i \in \{0, \dotsc ,k\! - \!1\}$. Notice that the induction hypothesis implies that if $z_i \in Y \cup Z$ then $|X_i| < |Y_i| - \sigma_P (Y_i)$ and if $z_i \in X \cup Z$ then $|X_i| \leq |Y_i| - b_i$. It also implies, in any case, $|X_i| \leq |Y_i| - \sigma_P (Y_i)$ and $|X_i| \leq |Y_i| - b_i$. Let now $z= z_{i+1}$, we have three possibilities:

\begin{enumerate}
\item If $z \in X$ then $X_{i+1} = X_i \cup \{z\}$, $Y_{i+1} = Y_i$, $\sigma_P(Y_{i+1}) = \sigma_P(Y_i)$ and $b_{i+1} = b_i$. Hence, in particular, $|X_{i+1} | = |X_i | +1$, since $P$ does not contain repetition of vertex. Morever, since $P(X) \subseteq Y$ (Definition \ref{DfAlternatingPath}), we have $z_i \in Y_i$ and so $|X_i| < |Y_i| - \sigma_P (Y_i)$ and $|X_i| < |Y_i| - b_i$ by induction hypothesis. Hence, obviously $|X_{i+1}| \leq |Y_{i+1}| - \sigma_P(Y_{i+1})$ and  $|X_{i+1}| \leq |Y_{i+1}| - b_{i+1}$.

\item If $z \in Y$ then $X_{i+1} = X_i$, $Y_{i+1} = Y_i \cup \{z\}$ and $b_{i+1} = b_i$. Hence,  $|Y_{i+1}| = |Y_i| +1$ and, since $|X_i| \leq |Y_i| - b_i$ by induction hypothesis, it comes easily $|X_{i+1}| < |Y_{i+1}| - b_{i+1}$. Now, if $z_i \in X$ then $z_iz \notin E(Y)$, and so $ \sigma_P (Y_{i+1}) = \sigma_P (Y_i)$. Hence, since $|X_i| \leq |Y_i| - \sigma_P (Y_i)$ by induction hypothesis, it comes $|X_{i+1}|   < |Y_{i+1}| - \sigma_P (Y_{i+1})$. If $z_i \in Y \cup Z$, we have $|X_i| < |Y_i| - \sigma_P(Y_i)$  by induction hypothesis. Moreover, we have either $\sigma_P(Y_{i+1}) = \sigma_P(Y_i) +1$ if $z_i \in Y$, or $\sigma_P(Y_{i+1}) = \sigma_P(Y_i)$ if $z_i \in Z$. Hence $\sigma_P(Y_{i+1}) \leq \sigma_P(Y_i) +1$ and so $|Y_i| - \sigma_P (Y_i) \leq |Y_{i+1}| - \sigma_P(Y_{i+1})$. It comes $|X_{i+1}| <  |Y_{i+1}| - \sigma_P (Y_{i+1})$.

\item If $z \in Z$ then we have $X_{i+1} = X_i$, $Y_{i+1} = Y_i$, $\sigma_P(Y_{i+1}) = \sigma_P(Y_i)$ and either $b_{i+1} = b_i +1$, if $Z_i = \emptyset$, or $b_{i+1} =b_i$ otherwise. Moreover, $z_i \notin X_i$, since $P(X) \subseteq Y$, and so $z_i \in Y \cup Z$. Hence we have $|X_i| < |Y_i| - \sigma_P (Y_i)$ by induction hypothesis and so it comes immediatly $|X_{i+1}| < |Y_{i+1}| - \sigma_P (Y_{i+1})$. If now $z_i \in Y$, we have $|X_i| < |Y_i| - b_i$ by induction hypothesis and, since $b_{i+1} \leq b_i +1$, it comes $|X_{i+1}| \leq |Y_{i+1}| - (b_i +1) \leq |Y_{i+1}| - b_{i +1}$. Finally, if $z_i \in Z$, we have $Z_i \neq \emptyset$ and so $b_{i+1} = b_i$. We have also $|X_i| \leq |Y_i| - b_i$ by induction hypothesis and so $|X_{i+1}| \leq |Y_{i+1}| - b_{i+1}$.
\end{enumerate}

\noindent Notice now that $P =P_k$, $X =X_k$, $Y =Y_k$ and $Z=Z_k$. Hence, since $z_k \in Y$ by hypothesis on $P$, we have $|X| < |Y| - \sigma_P (Y)$. Moreover, if $Z \neq \emptyset$ then $b_k =1$ and so $|X| < |Y| - 1$.
\end{proof}



\begin{lemma}\label{LmGeneralYXTer} A $Y \! X$-semi-alternating path $P$ is $Y \! X$-alternating if and only if $|X| \geq |Y|-1$.  
\end{lemma}
\begin{proof} From Fact \ref{FaAlternatingPath}, it is obvious that if $P$ is $Y \! X$-alternating then $P$ is $Y \! X$-semi-alternating and such that $|X| \geq |Y|-1$. Suppose now that $P$ is $Y \! X$-semi-alternating and $|X| \geq |Y|-1$. By Definition \ref{DfAlternatingPath}, it remains to show that $P(Y) \subseteq X$. It is done by contradiction. Indeed, suppose that there exists $y \in Y$ and $z \notin X$ such $yz \in \mathbb{E}(P)$. We have either $z \in Y$ and so $\sigma_P(Y) \geq 1$, or $z \in \mathbb{V}(P) \setminus X \cup Y$. In the first case, we get $|X| \geq |Y| -1 \geq |Y| - \sigma_P(Y)$, contradicting Lemma \ref{LmGeneralYX}.\ref{LmGeneralYX.1}. In the second case we get $|X| < |Y|-1$ by Lemma \ref{LmGeneralYX}.\ref{LmGeneralYX.2}, contradicting $|X| \geq |Y|-1$.
\end{proof}

\section{N2-closures and Hamiltonicity}\label{SecMain}

The notions of N-eligible vertex and N-closure of a graph were defined in \cite{BreVal3}. A vertex $x$ is N-eligible if it is non-simplicial and $|\bar x| \geq |N(\bar x)|$. The N-closure $cl_N(G)$ of $G$ is obtained by performing a local completion for every N-eligible vertex of $G$. The mains result of \cite{BreVal3} states that, for every graph $G$, $G$ is a spanning subgraph of $cl_N(G)$, $cl_N(G)$ does not contain any N-eligible vertex and $c(G) = c(cl_N(G))$.

In this section, we first introduce a generalization of the notion of N-eligibility called the \emph{2-weighted N-eligibility (N2-eligibility)}. Then, after having described some techniques for transforming semi-alternating paths into cycles, we show in Theorem \ref{ThmN2Closure} that the main result of \cite{BreVal3} can be extended to N2-closures of graphs.

\subsection{N2-eligibility: definition}

The N2-eligibility is defined using a kind of weight-function $\chi _2$, the weight depending of the number of edges in $E(N(\bar x))$. More precisely, $\chi _2$ counts the number of edges in $E(N(\bar x))$ up to $2$. It is easy to see that every N-eligible vertex is also N2-eligible.

\begin{definition} Let $G$ be a graph. We define the function $\chi_2 : V \mapsto \{0,1,2\}$ by, for every $x \in V$:

$\chi_2 (x) = \left\{
\begin{array}{lll}
0 & \texttt{if $ \sigma (N(\bar x)) = 0$}\\
1 & \texttt{if $ \sigma (N(\bar x))= 1$}\\
2 & \texttt{otherwise}
\end{array}
\right.$
\end{definition}

\begin{definition}\label{DfN2-Eligible} A vertex $x$ of $G$ is \emph{2-weighted N-eligible (N2-eligible)} if $x \in N\!S$ and $|\bar x| \geq |N(\bar x)| - \chi_2 (x)$. 
\end{definition}

\subsection{Building cycles from semi-alternating paths}

In this section, we assume a graph $G = (V,E)$, a N2-eligible vertex $x \in V$ and a $Y\!\bar x$-semi-alternating path $P$ in $G_x$ such that $Y \subseteq N(\bar x)$. We remind the reader that $G$ is a spanning subgraph of $G_x$. Notice that the graph $G$ of the previous section corresponds here to $G_x$.

\begin{fact}\label{FaBasicYBarXAlternating} Since $P$ is $Y\!\bar x$-semi-alternating and $Y \subseteq N(\bar x)$, it comes:
\begin{enumerate}
\item\label{FaBasicYBarXAlternating.1} $[Y,\bar x] \subseteq E$.
\item\label{FaBasicYBarXAlternating.2} $|Y| - \sigma_P(Y) >  |\bar x| \geq |Y| - \chi_2(x)$.
\item\label{FaBasicYBarXAlternating.3} $\chi_2(x) > \sigma_P(Y)$ and $\chi_2(x) \in \{1,2\}$.
\item\label{FaBasicYBarXAlternating.4} $\chi_2(x) - \sigma_P(Y) \in \{1,2\}$.
\end{enumerate}
\end{fact}
\begin{proof} Since $Y \subseteq N(\bar x)$, the first point comes immediatly from Fact \ref{FaBasicN-equivalence}.\ref{FaBasicN-equivalence.1} and we have $|\bar x| \geq |N(\bar x)| - \chi_2(x) \geq |Y| - \chi_2(x)$ by N2-eligibility of $x$. We have also $|\bar x| < |Y| - \sigma_P (Y)$ by Lemma \ref{LmGeneralYX}.\ref{LmGeneralYX.1}. That proves the second point which in turn implies $|Y| - \sigma_P(x) >  |Y| - \chi_2(x)$ and so $\chi_2(x) > \sigma_P(Y)$. We get $\chi_2(x) \neq 0$ and so the third point, which implies easily the fourth one.
\end{proof}

\begin{lemma}\label{LmFromAlternatingToC} If $P$ is $Y\!\bar x$-alternating then, for every path $Q$ in $G_x$ with distinct endpoints in $Y$ and without any other common vertex with $P$, there exists a cycle $C$ in $G_x$ such that $\mathbb{V}(C) = \mathbb{V}(P)\cup \mathbb{V}(Q)$ and $\mathbb{E}(C) \subseteq  \mathbb{E}(P) \cup [Y,\bar x] \cup \mathbb{E}(Q)$. 
\end{lemma}
\begin{proof} Let $P= y_0x_0 \dotsc x_{n-1} y_n$, where $y_0, \dotsc y_n$ is an enumeration of $Y$ and $x_0, \dotsc , x_{n-1}$ an enumeration of $\bar x$. Let $Q$ be a path satisfying the conditions of the lemma and let $y_i,y_j \in \mathbb{V}(P)\cap Y$, $i,j \in \{0, \dotsc ,n\}$, be the endpoints of $Q$, that is $Q = y_i \stackrel{\rightarrow}{Q}y_j$. We have $i\neq j$, since these endpoints are distinct, and we can suppose $i <j$. Notice that $[Y, \,\{x_i, x_j\}] \subseteq E$ (Fact \ref{FaBasicYBarXAlternating}.\ref{FaBasicYBarXAlternating.1}) and that $\mathbb{V}(Q) \cap \mathbb{V}(P) = \{y_i, y_j\}$ by hypothesis. Hence, if $j =n$ then $C=y_0 \stackrel{\rightarrow}{P} y_i \stackrel{\rightarrow}{Q} y_n\stackrel{\leftarrow}{P} x_iy_0$ is a cycle in $G_x$. If $j <n$ then $C=y_0 \stackrel{\rightarrow}{P} y_i \stackrel{\rightarrow}{Q} y_j \stackrel{\leftarrow}{P} x_i y_n \stackrel{\leftarrow}{P}x_j y_0$ is a cycle in $G_x$. Clearly, in both cases, $\mathbb{V}(C) = \mathbb{V}(P)\cup \mathbb{V}(Q)$ and every edge of $C$ not already in $P$ is either in $Q$ or in $[Y,\, \{x_i,x_j\}] \subseteq [Y, \bar x]$. Hence, $\mathbb{E}(C) \subseteq  \mathbb{E}(P) \cup  [Y, \bar x] \cup \mathbb{E}(Q)$.
\end{proof}

\begin{lemma}\label{LmN2CompletionCircumference} If $P$ is $Y\!\bar x$-alternating and there exists a cycle $C$ in $G_x$ of maximal length such that $\mathbb{V}(C) = \mathbb{V}(P)$ then there exists a cycle $C'$ in $G_x$ such that $\mathbb{V}(C') = \mathbb{V}(P)$ and $\mathbb{E}(C') \subseteq \mathbb{E}(P) \cup E$. 
\end{lemma}
\begin{proof} Suppose that $P$ is $Y\!\bar x$-alternating and let $C$ be a cycle in $G_x$ of maximal length such that $\mathbb{V}(C) = \mathbb{V}(P)$. Let $d = \chi_2(x) - \sigma_P(Y)$. We have $\sigma_P(Y) = \chi_2(x) - d$ and so, from Fact \ref{FaBasicYBarXAlternating}.\ref{FaBasicYBarXAlternating.2}, we have $|Y| -  \chi_2(x) + d >  | \bar x| \geq |N(\bar x) | - \chi_2(x)  \geq |Y| - \chi_2(x)$. Hence, in particular: $(1) \; \;   |Y| + d >   |N(\bar x) | \geq |Y|$.

We show now that there is an edge $uv \in E(Y)$ such that $uv \notin \mathbb{E}(P)$. If $d =1$ then there exists a unique edge $uv \in E(N(\bar x))$ such that $uv \notin \mathbb{E}(P)\cap [Y,Y]$. From $(1)$, we get $|Y| +1 > |N(\bar x) | \geq  |Y|$ and so, since $Y \subseteq N(\bar x)$, $Y = N(\bar x)$ and so $uv \in E(Y)$. If now $d =2$ then there exist two edges $uv,u'v' \in E(N(\bar x))$ such that $uv, u'v' \notin \mathbb{E}(P) \cap [Y,Y]$. If at least three distinct vertices among $u,v,u',v'$ belong to $Y$ then at least two belong to the same edge $uv$ or $u'v'$, and we get the result we are looking for. Suppose now that at most two vertices among $u,v,u',v'$ belong to $Y$. From $(1)$, we have $|Y| +2 > |N(\bar x) | \geq |Y|$. Hence there exists at most one vertex in $N(\bar x) \setminus Y$. So $uv$ and $u'v'$ must have a common vertex, otherwise $u,v,u',v'$ would be distinct vertices of $N(\bar x)$ and so at least three of them would be in $Y$. Without loss of generality, we can suppose $v =u'$. Notice that, since $|N(\bar x) \setminus Y| \leq 1$, at least two vertices among $u,v,v'$ are in $Y$ and so exactly two of them are. If they both belong to the same edge, we have the result we are looking for. We show now that the second case is impossible. Indeed, if we suppose $u,v' \in Y$, and since $v \notin \bar x \cup Y = \mathbb{V}(P)$ (Fact \ref{FaAlternatingPath}.\ref{FaAlternatingPath.3}), the path $Q = uvv'$ satisfies the conditions of Lemma \ref{LmFromAlternatingToC}. Hence, there exists a cycle $C'$ in $G_x$ such that $\mathbb{V}(C') = \mathbb{V}(P)\cup \mathbb{V}(Q)$. But $\mathbb{V}(P)\cup \mathbb{V}(Q) = \mathbb{V}(P) \cup \{v\}$ and since $\mathbb{V}(P) =  \mathbb{V}(C)$ by hypothesis, $C'$ is a cycle containing one more vertex than $C$, contradicting the maximality of $C$. 

Hence, we have shown that there exists an edge $uv \in E(Y)$ such that $uv \notin \mathbb{E}(P)$. It remains to notice that $Q = uv$ is a path satisfying the condition of Lemma \ref{LmFromAlternatingToC}. Hence there exists a cycle $C'$ such that $\mathbb{V}(C') = \mathbb{V}(P)\cup \mathbb{V}(Q)$ and $\mathbb{E}(C') \subseteq \mathbb{E}(P) \cup [Y,\bar x] \cup \{uv\}$. Since $ \mathbb{V}(P)\cup \mathbb{V}(Q)= \mathbb{V}(P)$ and $[Y ,\bar x] \cup \{uv\} \subseteq E$ (Fact \ref{FaBasicYBarXAlternating}.\ref{FaBasicYBarXAlternating.1} and $uv \in E(Y)$), $C'$ is the cycle we are looking for.
\end{proof}

\begin{lemma}\label{LmN2CompletionCircumferenceProper} If $P$ is a proper $Y\!\bar x$-semi-alternating path then there exists a cycle $C$ such that $\mathbb{V}(C) = \mathbb{V}(P)$ and $\mathbb{E}(C) \subseteq \mathbb{E}(P) \cup E$. 
\end{lemma}
\begin{proof} Suppose that $x$ and $P$ satisfy the conditions of the lemma. Since $P$ is proper, we have $|\bar x| < |Y| - 1$ by Lemma \ref{LmGeneralYXTer}. Hence $|Y | -2 \geq |\bar x|$ and since $\chi _2(x) \leq 2$, it comes also $|\bar x| \geq |Y| - \chi_2(x) \geq |Y | - 2$ by Fact \ref{FaBasicYBarXAlternating}.\ref{FaBasicYBarXAlternating.2}. Hence, we have $|\bar x| = |Y| -2$ and so $\chi_2(x) = 2$, otherwise we would have $\chi_2(x) =1$ by Fact \ref{FaBasicYBarXAlternating}.\ref{FaBasicYBarXAlternating.3}, and so $|Y| -2 = |\bar x| \geq |Y| -1$ by Fact \ref{FaBasicYBarXAlternating}.\ref{FaBasicYBarXAlternating.2}. From $Y \subseteq N(\bar x)$, we have also $Y = N(\bar x)$, otherwise we would have $|Y| -2 = |\bar x| \geq |N(\bar x) | - 2 > |Y| -2$. Let now $n = |Y|$, $m= |\bar x|$, $y_1, \dotsc y_n$ be an enumeration of the vertices of $Y$ ordered as they appear in $P$ and $x_1, \dotsc , x_m$ be a similar enumeration for $\bar x$. Since $P$ is not $Y\! \bar x$-alternating and $P(\bar x) \subseteq Y$ (Definition \ref{DfAlternatingPath}), there is a smaller $i \in \{1, \dotsc ,n \! - \! 1\}$ such that $y_i^+ \notin \bar x$. Let $P_0 = y_1 \stackrel{\rightarrow}{P} y_i$ and $P_1 = y_{i+1} \stackrel{\rightarrow}{P} y_n$. Let also $X_j = \mathbb{V}(P_j) \cap \bar x$ and $Y_j = \mathbb{V}(P_j) \cap Y$, $n_j = | Y_j|$ and $m_j = |X_j|$, where $j \in \{0,1\}$. Clearly, since $P(\bar x) \subseteq Y$, we have $P_j(X_j) \subseteq Y_j$ and so $P_j$ is $Y_j\!X_j$-semi-alternating, $j \in \{0,1\}$. By minimality of $i$, $P_0 = y_1 x_1 \dotsc x_{i-1} y_i$ and so $P_0$ is $Y_0\! X_0$-alternating. Moreover, since $y_{i+1}$ is the next element of $Y$ appearing in $P$ after $y_i$, there is no vertex of $Y$ in between $y_i$ and $y_{i+1}$. In addition, since $P(\bar x) \subseteq Y$ and $y_i^+ \notin \bar x$, it is easy to see there is no vertex of $\bar x$ in between $y_i$ and $y_{i+1}$. Hence, $\bar x = X_0 \cup X_1$ and $Y =Y_0 \cup Y_1$ and, since $P_0$ contains $i$ vertices of $Y$ and $i-1$ vertices of $\bar x$, we have $n_1 = n - i$ and $m_1 = n-2 - (i-1)$. It comes $m_1 = n_1-1$ and so by Lemma \ref{LmGeneralYXTer}, $P_1$ is the alternative sequence $y_{i+1} x_i \dotsc x_{n-2} y_n$.\smallskip 

We show now that there exists an edge $uv \in E(Y)$ such that $uv \notin \mathbb{E}(P)$ and $uv \neq y_iy_{i+1}$. Let $d =  \chi_2(x) - \sigma_P(Y)$. Since $d \geq 1$ (Fact \ref{FaBasicYBarXAlternating}.\ref{FaBasicYBarXAlternating.2}), there exists an edge $uv \in E(N(\bar x))$ such that $uv \notin \mathbb{E}(P) \cap [Y,Y]$. Since $Y = N(\bar x)$, we have $uv \in E(Y)$. If now $y_{i+1} = y_i^+$, that is, $y_iy_{i+1} \in \mathbb{E}(P)$ and $P =P_0P_1$, we have $uv \neq y_iy_{i+1}$, since $uv \notin \mathbb{E}(P)$, and so $uv$ is the edge we are looking for. Now if $y_{i+1}\neq  y_i^+$, since $y_{i+1}$ is the next vertex of $Y$ after $y_i$ on $P$ and since neither $P_0$ nor $P_1$ contains an edge in $[Y,Y]$, it is clear that $P$ contains no such edge. Hence $d =2 - 0 =2$ and there exists another edge $u'v' \in E(N(\bar x))$ such that $u'v' \notin \mathbb{E}(P) \cap [Y, Y]$. Since $Y = N(\bar x)$, we have $u'v' \in E(Y)$ and so at least one edge among $uv$ and $u'v'$ must be different from $y_{i}y_{i+1}$. Without loss of generality, we can suppose $uv$ this edge.\smallskip

Finally, we show that there exists a cycle $C$ in $G_x$ such that $\mathbb{V}(C) = \mathbb{V}(P)$ and $\mathbb{E}(C)  \subseteq \mathbb{E}(P) \cup E$. 
Since $[Y, \bar x] \cup \{uv\} \subseteq E$ (Fact \ref{FaBasicYBarXAlternating}.\ref{FaBasicYBarXAlternating.1} and $uv \in E(Y)$), it is sufficient to show that there is a cycle $C$ in $G_x$ such that $\mathbb{V}(C) = \mathbb{V}(P)$ and $\mathbb{E}(C) \subseteq  \mathbb{E}(P) \cup [Y, \bar x] \cup \{uv\}$. We define now $Q= \emptyset$ if $y_{i+1} = y_i^+$, and $Q = y^+_i \stackrel{\rightarrow}{P} y^-_{i+1}$ otherwise. We have $P = P_0QP_1$ where $P_0QP_1$ is defined as $P_0P_1$ if $Q = \emptyset$. Since $uv \in E(Y)$, there are $k,l \in \{1, \dotsc , n\}$ such that $uv =y_k y_l$. Without loss of generality we can suppose $k <l$, and so $k\neq n$ and $l\neq 1$. Moreover, we have $y_ky_l \neq y_iy_{i+1}$ since $y_ky_l = uv \neq y_iy_{i+1}$. Finally, since $P = y_1x_1 \dotsc x_{i-1}y_i Q y_{i+1}x_i \dotsc x_{n-2}y_n$ and $\mathbb{V}(Q) \cap Y = \emptyset = \mathbb{V}(Q) \cap X$, it is clear that the successor $y^+$ in $P$ of every vertex $y \in Y$ is in $\bar x$, except for $y_i$ and $y_n$. In particular $y_n^+$ is not defined. Similarly, the predecessor $y^-$ of every $y \in Y$ is in $\bar x$ except for $y_{i+1}$ and $y_1$, where $y_1^-$ is not defined. We make now two cases:

\begin{itemize}
\item Suppose first $k =i$ and so, since $y_ky_l \neq y_iy_{i+1}$, $i+1 < l \leq n$. We have $y_l \in Y_1$, $y_l^- \in X_1$ and so $[Y, \, \{y_l^-\}] \subseteq [Y, \bar x] \subseteq E$. Now if $y_k^- \notin \bar x$, we have either $k=1$ or $k = i+1$. The second case is impossible, since $k=i$, and so $k = 1$. Hence, $P_0 = y_1$ and $C = y_1y_l\stackrel{\rightarrow}{P} y_n y_l^- \stackrel{\leftarrow}{P} y_1$ is the cycle we are looking for. If now $y_k^- \in \bar x$ then $k \neq 1$, $[Y, \{y_k^- \}] \subseteq [Y, \bar x] \subseteq E$ and $C = y_1 y_l^- \stackrel{\leftarrow}{P} y_k y_l \stackrel{\rightarrow}{P} y_n y_k^- \stackrel{\leftarrow}{P} y_1$ is the cycle we are looking for.\smallskip

\item Suppose now $k \neq i$. We have $y_k^+ \in \bar x$ and so $[Y, \{y_k^+\}] \subseteq [Y,\bar x] \subseteq E$. If $y_l^- \in \bar x$ then $[Y, \{y_l^- \}] \subseteq [Y, \bar x] \subseteq E$ and so $C = y_1 \stackrel{\rightarrow}{P} y_k y_l \stackrel{\rightarrow}{P} y_n y_k^+\stackrel{\rightarrow}{P}y_l^-y_1$ is the cycle we are looking for. If $y_l^+ \in \bar x$ then $[Y, \{y_l^+\}] \subseteq [Y, \bar x] \subseteq E$ and $C= y_1 \stackrel{\rightarrow}{P} y_k y_l \stackrel{\leftarrow}{P} y_k^+ y_n \stackrel{\leftarrow}{P} y_l^+y_1$ is the cycle we are looking for. Finally, if $y_l^- \notin \bar x$ then $l = i+1$. If moreover $y_l^+ \notin \bar x$ then either $l=i$ or $l=n$. The case $l =i$ being impossible, we have $l = i+1 =n$, $P_1 = y_n$ and $C = y_1  \stackrel{\rightarrow}{P} y_k y_n \stackrel{\leftarrow}{P}y_k^+ y_1$ is the cycle we are looking for.
\end{itemize}
\end{proof}

\subsection{The N2-closure operation preserves Hamiltonicity}\label{SecMainResults}

Theorem \ref{ThmN2Closure} is essentially a corollary of the proposition below, which states that the circumferencee is preserved by local completion at a N2-eligible vertex. 

\begin{proposition}\label{PrN2CompletionCircumference} If $x$ is a N2-eligible vertex of $G$ then $c(G_x) = c(G)$.
\end{proposition}
\begin{proof} Since $G$ is a spanning subgraph of $G_x$, clearly $c(G) \leq c(G_x)$. Now to prove $c(G_x) \leq c(G)$ it is sufficient to prove that every cycle $C$ of $G_x$ can be transformed into a cycle $\mathbb{C}$ of $G$ such that $\mathbb{V}(\mathbb{C}) = \mathbb{V}(C)$. The proof is by induction on the number of egdes of $C$ which are in $B_x$. Indeed, if $C$ contains no edge of $B_x$ then $C$ is a cycle of $G$ and the result is immediate. Now suppose that $C$ contains $k+1$ egdes of $B_x$ and let $yz \in B_x \cap \mathbb{E}(C)$. By Definition \ref{DfLocalCompletion}, $y,z \in N(\bar x)$ and $yz \notin E$. Moreover, since $[\{y,z\}, \bar x] \subseteq E \subseteq E_x$ by Fact \ref{FaBasicN-equivalence}.\ref{FaBasicN-equivalence.1}, the maximality of $C$ implies $\mathbb{V}(C) \cap \bar x = \bar x$  by Lemma \ref{LmClassXonC}. 

Up to a rotation, we can suppose that $y$ is the starting and ending point of $C$, and so either $C = yz \stackrel{\rightarrow}{C}y$, or $C = y\stackrel{\rightarrow}{C}zy$. Without loss of generality, we can suppose the second case (otherwise the proof is done using $\stackrel{\leftarrow}{C}$). Let now $P = y \stackrel{\rightarrow}{C} z$. Clearly $\mathbb{V}(P)= \mathbb{V}(C)$ and $\mathbb{E}(P) = \mathbb{E}(C) \setminus \{yz\}$. Hence in particular, $| B_x \cap \mathbb{E}(P)| = k$ and $\bar x \subseteq \mathbb{V}(P)$. 

We show now that there exists a cycle $C'$ in $G_x$ such that $\mathbb{V}(P) = \mathbb{V}(C')$ and $\mathbb{E}(C') \subseteq \mathbb{E}(P) \cup E$. Let $Y = \mathbb{V}(P) \cap  N(\bar x)$. Clearly $Y$ and $\bar x$ are disjoint subsets of $\mathbb{V}(P)$ and $y,z \in Y$. Moreover, by Facts \ref{FaBasicN-equivalence}.\ref{FaBasicN-equivalence.3} and \ref{FaBasicLocalCompletion}.\ref{FaBasicLocalCompletion.1.5}, we have $N[\bar x] = N_x[x']$, for every $x' \in \bar x$, and so it is straightforward to check that $P(\bar x) \subseteq \bar x \cup Y$. Hence $P$ is a $Y\!\bar x$-pseudo-alternative path (cf. Definition \ref{DfAlternatingPath}). We have also $[Y, \bar x] \subseteq E$ by Fact \ref{FaBasicN-equivalence}.\ref{FaBasicN-equivalence.1}. Now, if $P$ is proper, that is, if $P(\bar x)\cap \bar x \neq \emptyset$, then there exists an edge $uv \in \mathbb{E}(P) \cap [\bar x,\bar x]$. Since $uz, vy \in [\bar x, Y] \subseteq E$, $C' = y \stackrel{\rightarrow}{P}uz \stackrel{\leftarrow}{P}vy$ is the cycle we are looking for. Now, if $P$ is $Y\!\bar x$-semi-alternating then there exists a cycle $C'$ such that $\mathbb{V}(C') = \mathbb{V}(P)$ and $\mathbb{E}(C') \subseteq \mathbb{E}(P) \cup E$ by Lemma \ref{LmN2CompletionCircumference} and \ref{LmN2CompletionCircumferenceProper}. 

Hence, in any case, there exists a cycle $C'$ such that $\mathbb{V}(P) = \mathbb{V}(C')$ and $\mathbb{E}(C') \subseteq \mathbb{E}(P) \cup E$. So, in particular, we have $B_x \cap \mathbb{E}(C') \subseteq B_x \cap \mathbb{E}(P)$ and the induction hypothesis applies to $C'$. Hence there is a cycle $\mathbb{C}$ of $G$ such that $\mathbb{V}(C') = \mathbb{V}(\mathbb{C})$. So, since $\mathbb{V}(C') = \mathbb{V}(P) = \mathbb{V}(C)$, $\mathbb{C}$ is the cycle of $G$ we are looking for. 
\end{proof}

\begin{figure}[!h]\label{N3EligibleBis}
\begin{center}
\psfrag{Titre1}{{\footnotesize{A non-Hamiltonian graph $G$}}}
\psfrag{Titre1Fin}{{\footnotesize where $x$ is N3-eligible}}
\psfrag{Titre 2}{{\footnotesize  The graph $G_x$ with Hamilton}}
\psfrag{Titre 2 Fin}{{\footnotesize cycle $xauwvcbyx$}}
\psfrag{x}{{\footnotesize$x$}}
\psfrag{a}{{\footnotesize$a$}}
\psfrag{b}{{\footnotesize$b$}}
\psfrag{c}{{\footnotesize$c$}}
\psfrag{1}{{\footnotesize$u$}}
\psfrag{2}{{\footnotesize$v$}}
\psfrag{3}{{\footnotesize$w$}}
\psfrag{y}{{\footnotesize$y$}}
\includegraphics[width=9cm]{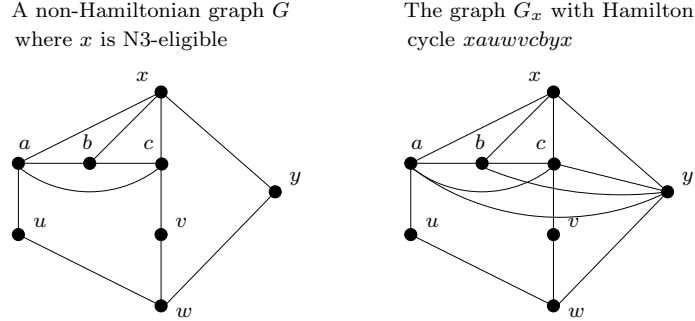} 
   \caption{A non-Hamiltonian graph with a Hamiltonian N3-completion.} 
   \protect
\label{N3EligibleWrongFig}   
\end{center}
\end{figure}  

Notice that the notion of N2-eligibility can be generalized to every positive integer $k$, in the obvious way, by using a weight-function $\chi_k$ which counts the edges of $E(N(\bar x))$ up to $k$. Nevertheless, the N2-eligibility is optimal in the sense that Proposition \ref{PrN2CompletionCircumference} is not always true as soon as $x$ is N$k$-eligible for some $k \geq 3$. A counter-example is given in Figure \ref{N3EligibleWrongFig} for $k=3$, and so for every $k \geq 3$, since every N$3$-eligible vertex is also N$k$-eligible for every $k > 3$.

We remind the reader that a \emph{choice function} on $V$ is a function $\rho : \mathcal{P}(V) \mapsto V$ such that, for every non-empty $X \in \mathcal{P}(V)$, $\rho(X) \in X$.

\begin{theorem}\label{ThmN2Closure} For all graph $G$ and choice function $\rho$ on $V$ there exists a graph $cl_{\rho}(G)$ containing no N2-eligible vertex and such that $G$ is a spanning subgraph of $cl_{\rho}(G)$ and $c(cl_{\rho}(G)) =c(G)$.
\end{theorem}
\begin{proof} The proof is by induction on the number of non-simplicial vertices of $G$. If $NS = \emptyset$ or if there is no N2-eligible vertex in $G$ then we define $cl_{\rho}(G) = G$. Otherwise, let $x = \rho (\nu)$, where $\nu$ is the set of N2-eligible vertices of $G$. Notice that $x \in N\!S$ and, since $\bar x \subseteq S_x \supseteq S$ by Fact \ref{FaBasicLocalCompletion}.\ref{FaBasicLocalCompletion.1.6}, it comes $x \in S_x \setminus S$, $S \subsetneq S_x$ and so $N\!S_x \subsetneq NS\!$. Hence, by induction hypothesis, there exists a graph $cl_{\rho}(G_x)$ which contains no N2-eligible vertex, such that $G_x$ is a spanning subgraph of $cl_{\rho}(G_x)$ and such that $c(G_x)=c(cl_{\rho}(G_x))$. Since $G$ is a spanning subgraph of $G_x$ and $c(G) = c(G_x)$ by Proposition \ref{PrN2CompletionCircumference}, the result comes by letting $cl_{\rho}(G)= cl_{\rho}(G_x)$.
\end{proof}

\begin{corollary} $cl_{\rho}(G)$ is Hamiltonian if and only if $G$ is, if and only if $cl_{\rho'}(G)$ is, for every choice function $\rho '$ on $V$.
\end{corollary}

\begin{figure}[!h]\label{CaptureN2Eligible}
\begin{center}
\includegraphics[width=13cm]{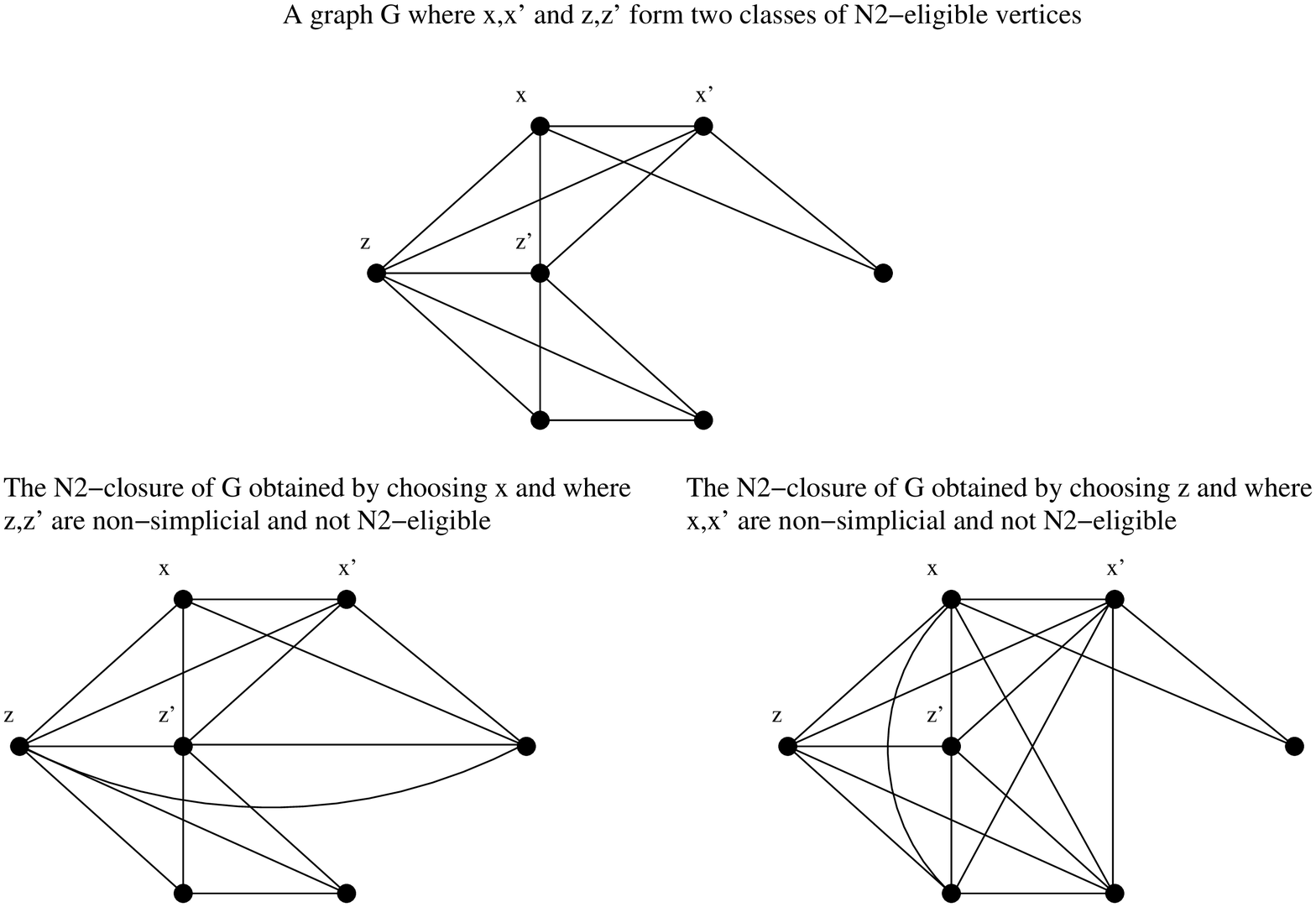} 
   \caption{A graph with two distinct non-optimal N2-closures.} 
\protect
\label{CaptureN2EligibleFig}   
\end{center}
\end{figure}


\section{Conclusion}

In this article, we introduced another closure concept preserving Hamiltonicity which is essentially a generalization of N-closure defined in \cite{BreVal3}. Nevertheless, due to its greater generality, the N2-closure obtained in Theorem \ref{ThmN2Closure} by recursively choosing a N2-eligible vertex $x$ may depend of the choice of $x$. Hence there are often more than one N2-closure for a given graph. As shown in Figure \ref{CaptureN2EligibleFig} below, this can be due to the fact that a non-simplicial vertex may be N2-eligible in $G$ but not in $G_x$, although it is still non-simplicial in $G_x$. Hence, contrary to the N-closure, the N2-closure is not optimal in the sense that every N2-eligible vertex of $G$ would be simplicial in $cl_{\rho}(G)$. Nevertheless, it seems that a strategy to build an optimal N2-closure is possible.

\end{document}